\documentclass[10pt,a4paper]{amsart}
\usepackage{amscd,amssymb,amsopn,amsmath,amsthm,graphics,amsfonts,enumerate,verbatim,calc}
\usepackage[dvips]{graphicx}

\usepackage{amssymb,amsmath, color}

\headsep=1cm \numberwithin{equation}{section}
\newtheorem{theorem}{Theorem}[section]
\newtheorem{lemma}[theorem]{Lemma}

\newtheorem{proposition}[theorem]{Proposition}
\newtheorem{corollary}[theorem]{Corollary}

\theoremstyle{definition}
\newtheorem{definition}[theorem]{Definition}
\theoremstyle{remark}
\newtheorem{remark}[theorem]{Remark}
\newtheorem{fact}[theorem]{Fact}
\newtheorem{example}[theorem]{Example}
\newtheorem{observation}[theorem]{Observation}
\newtheorem{discussion}[theorem]{Discussion}

\newtheorem{acknowledgement}{Acknowledgement}

\newcommand{\Se}{\operatorname{S}}
\newcommand{\loc}{\operatorname{Loc}}
\newcommand{\UFD}{\operatorname{UFD}}
\newcommand{\im}{\operatorname{im}}

\newcommand{\grade}{\operatorname{grade}}

\newcommand{\Aut}{\operatorname{Aut}}
\newcommand{\cgrade}{\operatorname{c.grade}}
\newcommand{\Cgrade}{\operatorname{\check{C}.grade}}
\newcommand{\cmdef}{\operatorname{CM.def}}
\newcommand{\Spec}{\operatorname{Spec}}

\newcommand{\ass}{\operatorname{Ass}}
\newcommand{\rad}{\operatorname{rad}}
\newcommand{\cd}{\operatorname{cd}}

\newcommand{\tr}{\operatorname{tr}}
\newcommand{\nr}{\operatorname{n}}

\newcommand{\Ht}{\operatorname{ht}}

\newcommand{\HH}{\operatorname{H}}

\newcommand{\V}{\operatorname{V}}

\newcommand{\Char}{\operatorname{char}}

\newcommand{\Supp}{\operatorname{Supp}}

\newcommand{\Hom}{\operatorname{Hom}}
\newcommand{\Syl}{\operatorname{Syl}}

\newcommand{\f}{\operatorname{f}}
\newcommand{\depth}{\operatorname{depth}}

\newcommand{\Cdepth}{\operatorname{\check{C}.depth}}
\newcommand{\GL}{\operatorname{GL}}
\newcommand{\Jac}{\operatorname{Jac}}

\newcommand{\lo}{\longrightarrow}
\newcommand{\fm}{\mathfrak{m}}
\newcommand{\fp}{\mathfrak{p}}
\newcommand{\fq}{\mathfrak{q}}
\newcommand{\fa}{\mathfrak{a}}
\newcommand{\fb}{\mathfrak{b}}

\newcommand{\fn}{\frak{n}}

\begin{document}

\author[]{Mohsen Asgharzadeh }

\address{}
\email{mohsenasgharzadeh@gmail.com}

\title[ ]
{ on the local cohomology of modular invariants}

\maketitle

\smallskip

\begin{abstract} We compute some  numerical invariants of local cohomology
of the ring
of invariants by a finite group, mainly in the modular case.
Also, we present some applications. In particular, we study Cohen-Macaulay
property of modular invariants from the viewpoints of depth, Serre's condition and the relevant generalizations
(e.g., the Buchsbaum property, etc). The situation in the  local case is different from the global case.
\end{abstract}

\section{Introduction}

There are a lot of research papers  on local cohomology modules
and on modular invariant theory. Also, there are some links between these. For example,
Ellingsrud and Skjelbred \cite{elli} (resp. Fogarty \cite{fog})
applied beautiful sort of local cohomology arguments to compute  bounds  on the depth of modular invariant rings.
We start by a new presentation of a result of  Kemper  \cite{k} (see Proposition \ref{kem}).
Then, we observe that:

\begin{observation}
Let
$G\to \GL(n, \mathbb{F})$ be a representation of a finite  group $G$  and denote the invariant ring by $R$.
Let $\fa\lhd_h R$ be such that $\HH^i_{\fa}(R)$ is of finite length for all $i<\cd(\fa)$. Then $\HH^i_{\fa}(R)=0$ for all $i<\cd(\fa)$. In particular,
$\grade(\fa,R)=\Ht(\fa)=\cd(\fa)=\f_{\fa}(R).$
\end{observation}

In dimension $4$ we can talk a little more (see Corollary \ref{ll4}). Almkvist and  Fossum proved that
$R:=\mathbb{F}_2[x_1,\ldots,x_4]^{\mathbb{Z}/4\mathbb{Z}}$ does not satisfy Serre's condition $\Se(3)$.
This is the Bertin's famous  ring which is not Cohen-Macaulay.
 In particular,
we can recover the following funny result of  Larry Smith:

\begin{corollary}Let
$G\to \GL(n, \mathbb{F})$ be a representation of a finite  group $G$  and denote the invariant ring by $R$. If $R$
is $\Se(n-1)$, then $R$  is Cohen-Macaulay. In particular, if $n=3$ then
 $\mathbb{F}[x,y,z]^G$ is Cohen-Macaulay.
\end{corollary}

 Also,
the following is  a strengthening a result  of Hartshorne and Ogus:

\begin{corollary}Let
$G\to \GL(5, \mathbb{F}_p)$ be a representation of a finite  $p$-group $G$  and denote the invariant ring by $R$. If $R$
is $\Se(3)$, then $R$  is Cohen-Macaulay (and hence Gorenstein).
\end{corollary}

The situation in   local case is quite different:
Despite of Corollary 1.2,
Fogarty constructed a regular local ring $R$ of dimension $3$ and of prime characteristic
equipped with  a wild action of a finite group $G$ such that  $\depth(R^G)=2$ (see  \cite{fog}).
By applying Fogarty's idea we present the following result:

\begin{observation}\label{frsh}
Let $R$ be a
local ring of  characteristic $p>0$. Let  $G$ be a cyclic group of
order $p^nh$, $n> 0$, acting wildly on $R$. The following assertions hold:
\begin{enumerate}
\item[i)]  If  $R$ is generalized Cohen-Macaulay, then   $R^G$ is generalized Cohen-Macaulay.
\item[ii)]   If  $\depth(R)>1$, then $\Ht(\im(R\stackrel{\tr}\lo R^G))=\dim R$.
\end{enumerate}
\end{observation}

 Let $G\to \GL(n, \mathbb{F}_p)$ be a representation of a finite group.
 It  is known  that $\Ht(\im(R\stackrel{\tr}\lo R^G))<\dim (\mathbb{F}_p[V]^G)$. This result is due to  Feshbach (see \cite[Theorem 6.4.7]{ns}).
Observation \ref{frsh} shows that this is not the case  for the wild actions over local rings of depth bigger than one.
It may be nice to characterize cases
for which $\im(R\stackrel{\tr}\lo R^G)$ is the maximal ideal. Here, we present a sample:

\begin{corollary}
Let $R$ be a $3$-dimensional Cohen-Macaulay
local ring of  characteristic $2$. Let  $G$ be a  group of
order $2$  acting wildly on $R$. Then $(R^G,\fn)$ is Buchsbaum if and only if $\im(\tr)=\fn$.
\end{corollary}

Also, we extend and simplify some results of Ellingsrud and Skjelbred by  elementary methods:

\begin{corollary}
Let $(R,\fm)$ be a    local domain and  $G$  a finite group  acting wildly on $R$.
\begin{enumerate}
\item[i)] If $R$ is $\Se(2)$ then $R^G$ is $\Se(2)$.
\item[ii)]  Suppose  $\depth(R)>1$ and $\dim R>2$. Then   $R^G$ is not $\Se(n)$ for  all $n>2$.
\end{enumerate}
\end{corollary}

\section{Preliminary}

\textmd{Subsection 2.A:} A quick review of local cohomology.
Let $R$ be any commutative  ring with an ideal $\mathfrak a$ with a generating set
 $\underline{a}:=a_{1}, \ldots, a_{r}$.
 By $\HH_{\underline{a}}^{i}(M)$, we mean the $i$-th cohomology of the
\textit{$\check{C}ech$} complex of  a module $M$ with respect  to $\underline{a}$. This is independent of the choose
of the generating set. For simplicity, we denote it by $\HH_{\mathfrak a}^{i}(M)$.  Recall that    $\cd(\fa,M)$ is defined to
be the spermium of $i$'s such that $\HH^i_{\fa}(M)\neq 0$. We set $\grade(\fa,M):=\inf\{i:\HH^{i}_{\fa}(M)\neq 0\}$. We put $\depth(M):=\grade(\fm,M)$ when
$\fm$ is the unique (graded) maximal ideal. We say $R$ satisfies  Serre's condition $\Se(n)$ if $\depth(R_{\fp})\geq\min\{n,\dim R_{\fp}\}$ for all $\fp\in\Spec(R)$.
Recall that $\f_{\fa}(M):=\inf\{i:\HH^{i}_{\fa}(M) \text{ is not finitely
generated}\}$. Here, is the relations between these:$$\grade(\fa,R)\leq\f_{\fa}(R)\leq\Ht(\fa)\leq\cd(\fa)\leq\dim R,$$
provided $R$ is noetherian. The term local means noetherian and quasi-local.

\begin{fact}\label{gro}The following  assertions hold:
\begin{enumerate}
\item[i)]  Let $(A,\fm)$ be a local ring. Then $\HH^{\dim A}_{\fm}(A)$ is not finitely generated.
\item[ii)]  (Grothendieck's finiteness theorem) Let $A$ be a quotient of a noetherian regular  ring. Then $$\f_{\fa}(A)=\inf\{\depth(A_{\fq})+\Ht(\frac{\fa+\fq}{\fq}):\fq\in\Spec(A)\setminus\V(\fa)\}.$$
\item[iii)]  (Hartshorne-Lichtenbaum vanishing theorem) Let $A$ be analytically irreducible local ring and $\dim (A/\fa)>0$. Then $\HH^{\dim A}_{\fa}(A)=0$.
 \end{enumerate}
\end{fact}
For more details on local cohomology modules see the books \cite{sga2}, \cite{41} and \cite{BS}.

\textmd{Subsection 2.B:} A quick review of local algebra.  A finite sequence $\underline{x}:=x_{1},\ldots,x_{r}$ of
elements of $(A,\fm)$ is called weak regular sequence if $x_i$ is a
nonzero-divisor on $A/(x_1,\ldots, x_{i-1})$ for $i=1,\ldots,r$.
Then $\depth(A)$ coincides with the supremum of the
lengths of all weak regular sequences contained in $\fm$.
Recall that a local ring $(A,\fm)$ is called Cohen-Macaulay if $\depth A=\dim A$.
Also,  $(A,\fm)$ is  called generalized Cohen-Macaulay if
$\ell(\HH^{i}_{\fm}(A))<\infty$ for all $i<\dim A$.
Recall that a sequence $x_{1},\ldots ,x_{r} \subset \fm$ is called a weak sequence
 if $\fm ((x_{1},\ldots ,x_{i-1})\colon x_{i})\subseteq(x_{1},\ldots ,x_{i-1})$ for all $i$.
The ring $R$ is called \textit{Buchsbaum} if every system of parameters is a weak sequence.
This implies that
$\fn\HH^i_{\fm}(R)=0$ for all $i<\dim R^G$. This property is called quasi-Buchsbaum. In general quasi-Buchsbaum
rings are not Buchsbaum. Despite of this, there are cases for which we can deduce the Buchsbaum property from  quasi-Buchsbaum:

\begin{fact}\label{qb}
Suppose $R$ is quasi-Buchsbaum and that $\HH^i_{\fm}(R)=0$ for all $i\neq\depth (R)$ and $i\neq\dim (R)$. Then is Buchsbaum.
\end{fact}

For more details on Buchsbaum rings (resp. local algebra), see the book \cite{bus} (resp. \cite{Mat}).

\textmd{Subsection 2.C:} A quick review of invariant theory.
Let $A:=\bigoplus_{n\geq0}A_n$ be the  polynomial algebra of a field  $\mathbb{F}$.
Suppose $G$ acts   on $A$ by degree-preserving  homomorphisms.
This means that $g(A_n) \subseteq A_n$ for all $g \in G$ and $n \in \mathbb{N}$. Then $$R:=A^G = \{a \in A : g(a)=a \quad \forall g \in G\}$$ is the (graded) ring of invariants.
The notation $\fa\lhd_hR$ stands for homogeneous ideal of $R$  and $\fm:= \bigoplus_{n>0}R_n$  is the irrelevant ideal.
Also, if $G$ is a finite group, then by the very first result of invariant theory,  $R$ is  finitely generated as an $R_0$-algebra (see \cite[Theorem 2.1.4]{ns}).
This  remarkable result is due to \textit{Emmy Noether}.
If $|G|$ is invertible in $R$ the action is called non-modular.
If $|G|$ is not invertible in $R$ the action is called \textit{modular}.

\begin{fact}
Let $S$ be a normal domain and $\mathcal{G}\subset\Aut(S)$ be any group. Then $S^{\mathcal{G}}$ is normal.
\end{fact}
 Let $S$ be any ring and $\mathcal{G}\subset\Aut(S)$ be a finite group. There is an  operation called transfer $\tr:S \lo S^\mathcal{G}$. It sends $a\in S$
to $\sum_{g\in G} ga$.
 If  $S$ is  a domain then $\tr$ is nonzero. This follows by a result of Dedekind (see \cite[Lemma 2.2.3]{ns}).
If $(S,\fn,k)$ is a
local ring, Fogarty  remarked   that $\im(\tr)$ is  proper  provided $\mathcal{G}$ acts trivially on $k$.

\begin{fact}\label{max}
Let $(S,\fm)$ be quasilocal and $\mathcal{G}\subset\Aut(S)$ be a finite group. Then $S^{\mathcal{G}}$ is quasilocal.
\end{fact}

\begin{proof}
 Let $\fn_1$ and $\fn_2$ be two maximal ideals of $S^{\mathcal{G}}$. The extension $S^{\mathcal{G}}\to S$ is integral. By lying over
 there are $P_i$ lying over $\fn_i$. Since $\fn_i$ is maximal, $P_i=\fm$ (see \cite[Lemma 2]{Mat}). So, $\fn_1=\fm\cap S^{\mathcal{G}}=\fn_2$.
 \end{proof}
For more details
on invariant theory of finite groups see the  book \cite{ns} by  Neusel and Smith.

\section{the global results}

 \begin{discussion}\label{dis}  (Kemper) Let
 $G\to \GL(V, \mathbb{F})$ be a representation of a finite group $G$ and denote the invariant ring by $R$. By Cohen-Macaulay defect we mean $\cmdef(R):=\dim(R)-\depth(R)$.
 For each $m$, let  $\loc(\cmdef>m):=\{\fp:\cmdef(R_{\fp})>m\}.$ In view of \cite[Proposition 3.1]{k}, $$0<\dim\left(\loc(\cmdef>m)\right)<\dim (V)-m-1,$$
provided $\cmdef(R)>m$.
 \end{discussion}
Set $X:=\Spec(-)$ and denote the punctured spectrum by  $\widehat{X}:=X\setminus\{\fm\}$. In general, it is not true to extend a property $\mathcal{P}$ from $\widehat{X}$  to  $X$. For example,  there are 3-dimensional normal rings that are not Cohen-Macaulay. In particular, $\mathcal{P}:=\textit{Cohen-Macaulay}$ does not extend
from $\widehat{X}$ to  $X$.

\begin{proposition} \label{kem}Let
$G\to \GL(n, \mathbb{F})$ be a representation of a finite group $G$ and denote the invariant ring by $R$.  The following are equivalent:
 \begin{enumerate}
\item[$\mathrm{(i)}$]   $R$ is Cohen-Macaulay,
\item[$\mathrm{(ii)}$] $R$ is Buchsbaum,
\item[$\mathrm{(iii)}$] $R$ is quasi-Buchsbaum,
\item[$\mathrm{(iv)}$] $R$ is generalized Cohen-Macaulay,
\item[$\mathrm{(v)}$] $R$ is Cohen-Macaulay over the punctured spectrum.
 \end{enumerate}
\end{proposition}

\begin{proof}
The implications $\mathrm{(i)}\Rightarrow\mathrm{(ii)}\Rightarrow\mathrm{(iii)}\Rightarrow\mathrm{(iv)}\Rightarrow\mathrm{(v)}$ are true over any commutative noetherian rings.
The implications $\mathrm{(ii)}\Rightarrow\mathrm{(i)}$ is due to Kemper. Let us repeat its argument to deduce  $\mathrm{(i)}$ from  $\mathrm{(v)}$:
Suppose $R$ is  generalized   Cohen-Macaulay. We are going to show it is $R$ is Cohen-Macaulay.
Suppose on the contradiction that $R$ is not Cohen-Macaulay. Then $\cmdef(R)>0$. We apply Discussion \ref{dis} for $m=0$
to find a prime $\fp$ of height $n-1$ such that $\fp\in\loc(\cmdef>0)$. This means that $R_{\fp}$ is not Cohen-Macaulay. This contradicts
the Cohen-Macaulay property over $\widehat{X}$.
 \end{proof}

 The module version of Proposition  \ref{kem} is not true:
Let $G$ be the trivial group act on $R:=k[x,y,z,w]$. So, $R^G=R$. In view of \cite{eg} there is a 
 prime ideal $\fp$ of height two such that $\HH^1_{\fm}(R/\fp)=R/\fm$. Set $M:=R/ \fp$. Since $\fp$ is prime,
$\HH^0_{\fm}(R/\fp)=0$. Note that $\dim M=2$. So, $M$ is generalized Cohen-Macaulay. Clearly, $M$ is not Cohen-Macaulay.

\begin{corollary}Let
$G\to \GL(n, \mathbb{F})$ be a representation of a finite  group $G$  and denote the invariant ring by $R$. If $R$
is $\Se(n-1)$, then $R$  is Cohen-Macaulay. In particular, if $n=3$ then
 $\mathbb{F}[x,y,z]^G$ is Cohen-Macaulay\footnote{Smith's argument is as follows:
First, he reduces to the prime characteristic case
 by applying the famous theorem of  Hochster and Eagon.
Second, he uses the following  result:
 If $\mathbb{F}[X,\ldots]^{\Syl_p(G)}$ is Cohen-Macaulay, then $\mathbb{F}[X,\ldots]^G)$
is Cohen-Macaulay (this is the main result of
 "Rings of invariants and p-Sylow subgroups" by Campbell, Hughes, and Pollack).
 Finally, he applies the mentioned result of Ellingsrud and Skjelbred
  to deduce that $\depth(\mathbb{F}[X,Y,Z]^{\Syl_p(G)})\geq3$.}.
\end{corollary}

\begin{proof}
Suppose on the contradiction that $R$ is not Cohen-Macaulay, e.g.,
$\cmdef(R)>0$. In the light of Discussion \ref{dis}, $\dim\left(\loc(\cmdef>0)\right)\neq 0$.
There is $\fp$ of height $n-1$ such that $\cmdef(R_{\fp})>0$ and so $\depth(R_{\fp})<n-1= \min\{n-1,\dim R_{\fp}\}$. This
contradicts $\Se(n-1)$.

Now we prove the particular case:
By a folklore result of Serre, noetherian normal rings are $\Se(2)$ (see \cite[Theorem 23.8]{Mat}). The claim follows by the first part.
\end{proof}

Recall that $G$ is called $p$-group if $|G|=p^i$ for some $i$.

\begin{corollary}Let
$G\to \GL(n, \mathbb{F}_p)$ be a representation of a finite  $p$-group $G$  and denote the invariant ring by $R$. Suppose $n<6$. If $R$
is $\Se(3)$, then $R$  is Cohen-Macaulay and hence Gorenstein.
\end{corollary}

\begin{proof}
We deal with the case $n=5$.
Suppose on the contradiction that $R$ is not Cohen-Macaulay. In view of \cite[Proposition 5.6.10]{ns}, $\depth(R)\geq3$. Suppose first that $\depth(R)=3$. Then $\cmdef(R)=2$.
In the light of Discussion \ref{dis}, $\dim\left(\loc(\cmdef>1)\right)\neq 0$. There is $\fp$ of height $4$ such that $\cmdef(R_{\fp})>1$. Apply $\Se(2)$ along this to see $2\leq\depth(R_{\fp})\leq4-2=2$, i.e., $\depth(R_{\fp})=2< \min\{3,\dim R_{\fp}\}$. This
contradicts $\Se(3)$. Suppose now that $\depth(R)\geq4$. Since $R$ is not Cohen-Macaulay, we deduce that $\depth(R)=4$.  Since $R$ is $\Se(3)$ we have $\depth (R_{\fp})\geq\frac{\Ht(\fp)}{2}+1$ for all $\fp\in \widehat{X}$. Since $\depth(R)=4$ we have  $\depth (R_{\fm})\geq\frac{\Ht(\fm)}{2}+1$. In sum,  $$\depth (R_{\fp})\geq\frac{\Ht(\fp)}{2}+1\geq\min\{\frac{\Ht(\fp)}{2}+1,\dim R_{\fp}\}$$ for all $\fp\in\Spec(R_{\fm})$.
Hartshorne and Ogus call this the \textit{C-property}. Now  we use \cite[Corollary 1.8]{HO}:

\begin{enumerate}
\item[Fact A)] Let $A$ be a local factorial ring  which is quotient of a regular local ring  with C-property. Then  $A$ is Cohen-Macaulay.
 \end{enumerate} Recall that invariant rings of $p$-groups are $\UFD$ (see \cite[Corollary1.7.4]{ns}).
Also, recall that a domain is $\UFD$ if and only if its height one prime ideals are principal. From this,
$R_{\fm}$ is $\UFD$. By Fact A), we get that $R_{\fm}$ is Cohen-Macaulay. Thus, $R$ is Cohen-Macaulay. This contradiction completes the proof.
\end{proof}

A ring $R$ is called
almost Cohen-Macaulay if $\grade(\fp, R)=\depth(\fp {R_\fp}, R_{\fp})$ for every $\fp\in\Spec(R)$.
We use the following
characterization of almost Cohen-–Macaulay rings: $R$ is an almost Cohen-–Macaulay ring if and only if $\Ht(\fp)\leq1+\grade(\fp, R)$ for every $\fp\in\Spec(R)$.

\begin{corollary}\label{nal} Let
$G\to \GL(n, \mathbb{F})$ be a representation of a finite group $G$ and denote the invariant ring by $R$. The following are equivalent:
 \begin{enumerate}
\item[$\mathrm{(i)}$]   $R$ is almost Cohen-Macaulay,
\item[$\mathrm{(ii)}$] $R$ is almost  Cohen-Macaulay over the punctured spectrum.
 \end{enumerate}In particular, if
$G\to \GL(4, \mathbb{F})$ then $R$
 is almost  Cohen-Macaulay and $\depth(R)\geq3$.\footnote{Let $R:=\mathbb{F}_2[x_1,\ldots,x_4]^{\mathbb{Z}/4\mathbb{Z}}$.
  It may be worth  to note that computing depth of $R$ was a challenging problem. In their paper Fossum and Griffith wrote that
 "many hours of calculations, using several hundred sheets of paper, have convinced the authors that the depth of Bertin's example is 3."
 (see page 194 of: \emph{Complete local factorial rings which are not Cohen-Macaulay in characteristic p},
  Ann. Sci. Ecole Norm. Sup. (4) \textbf{8} (1975),  189–-199). Now, we know that there are at least four different arguments for this:
 The first proof of this is in \cite[Proposition 2.3]{af}. They worked with Bertin's example, directly. The second one is in \cite[Corollaire 3.3]{elli}.
 This is more general than the first one: They work with indecomposable action of $\mathbb{Z}/p^n\mathbb{Z}$.
  The third one is  Corollary \ref{nal}
 which works for any 4-dimensional invariant rings. The fourth one is in the book \cite[Proposition 5.6.10]{ns}.
  This is very strong: for an invariant ring $A$ by a finite group we have
$\depth(A)\geq\min\{3,\dim A\}$. Concerning depth of Bertin's ring, I feel that  there
 is a typographical error in the book "polynomial invariants of finite groups" by Benson.}
\end{corollary}

\begin{proof}
Suppose $R$ is almost Cohen-Macaulay. This is clear from definition that  $R$ is almost  Cohen-Macaulay over $\widehat{X}$.
Conversely, suppose  $R$ is almost   Cohen-Macaulay over $\widehat{X}$. 
  On may find easily that $\depth(R_{\fm})=\depth(R)$ and  $\dim(R_{\fm})=\dim(R)$. Suppose on the contrary that
$R$ is not almost   Cohen-Macaulay. One may read this
as $\cmdef(R)>1$.
In view of Discussion \ref{dis}, $\dim\left(\loc(\cmdef>1)\right)\neq 0$.
In particular, there is a height $n-1$ prime ideal $\fp$ such that $\dim(R_{\fp})-\depth(R_{\fp})>1$.
This contradicts almost Cohen-Macaulayness  over $\widehat{X}$.

Any four-dimensional normal domain is almost  Cohen-Macaulay over the punctured spectrum. This follows by  Serre's characterization
 of normality. Conjugate this
along with the first item to check almost  Cohen-Macaulayness of $R$. In particular, $\depth(R)\geq\dim R-1=3$.
\end{proof}

Here, $\ell(-)$ is the length function.

\begin{proposition}\label{def2}Let
$G\to \GL(n, \mathbb{F})$ be a representation of a finite  group $G$  and denote the invariant ring by $R$.
 Let $\fa\lhd_h R$ be such that $\ell(\HH^i_{\fa}(R))<\infty$  for all $i<\cd(\fa)$. Then $\HH^i_{\fa}(R)=0$ for all $i<\cd(\fa)$. In particular,
$\grade(\fa,R)=\Ht(\fa)=\cd(\fa)=\f_{\fa}(R).$
\end{proposition}

\begin{proof}
We have $\dim R=n$. Note that $\grade(\fa,R)\leq\Ht(\fa)\leq\cd(\fa,R)\leq \dim R$. We claim that $\Ht(\fa)=\cd(\fa)$: Indeed, let $\fp\in\min(\fa)$ of same height as of $\fa$.
By Fact \ref{gro} $\HH^{\dim(R_{\fp})}_{\fp R_{\fp}}(R_{\fp})$
 is not finitely generated. Since
 $\HH^{\Ht(\fp)}_{\fp R_{\fp}}(R_{\fp})\simeq\HH^{\Ht(\fa)}_{\fa R_{\fp}}(R_{\fp})\simeq(\HH^{\Ht(\fa)}_{\fa}(R))_{\fp}$ we  get that  $\HH^{\Ht(\fa)}_{\fa}(R)$
 is not finitely generated.  We look at our assumption to observe that $\Ht(\fa)=\cd(\fa)$,
  as claimed. In sum,  $\grade(\fa,R)\leq\Ht(\fa)=\cd(\fa,R)$.
In particular, if $\grade(\fa,R)= \Ht(\fa),$ then  by definition of grade,
$\HH^i_{\fa}(R)=0$ for all $i<\cd(\fa)$. This yields the claim in the Cohen-Macaulay case. Let $m\geq 0$ be such that
$\dim(R)-\depth R=m+1$. We bring the following fact:

Fact A):  Let $M$ be any graded $R$-module (not necessarily finitely generated)
such that $r:=\grade(\fa,M)<\depth M$. Then $\HH^{r}
_{\fa}(M)$ is not artinian. This stated in \cite{asg} without any proof. We use induction on $r$ to prove it. Let $r=0$. If $\HH^{0}
_{\fa}(M)$ were be artinian, then we should have $0\neq\HH^{0} _{\fa}(M)\cong
\HH^{0}_{\fm}(\HH^{0} _{\fa}(M) )\cong \HH^{0}_{\fm}(M).
$ In particular,  $\HH^{0}_{\fm}(M)\neq 0$. That is $\depth(M)=0$. This
contradicts $r<\depth M$. Now, assume that $r>0$.
Recall that $\grade(\fa,M):=\inf\{i:\HH^{i}_{\fa}(M)\neq 0\}$.
In particular, $\HH^0_{\fa}(M)=0$. Let $E$ be the graded injective
envelope of $M$ and $N:=E/M$ (such a thing exists. For more details   see \cite[\S 12]{BS}). Recall that $M\subset E$ is graded-essential if $M\cap F	\neq0$ for every non-zero graded submodule $F\subset E$.
Let $e\in \HH^0_{\fa}(E)$ by any graded element.  Suppose $e$ is not zero. Since the extension $M\subset E$ is essential, there is a homogeneous element
$r\in R$ such that $0\neq re\in M$. Clearly,  $re\in \HH^0_{\fa}(M)=0$. This contradiction says that $\HH^0_{\fa}(E)=0$. Similarly,
$\HH^0_{\fm}(E)=0$. From
$0\to M\to E\to N\to0,$ we obtain that $\HH^{i} _{\fa}(N)\cong \HH^{i+1}
_{\fa}(M)$ and $\HH^{i} _{\fm}(N)\cong \HH^{i+1} _{\fm}(M)$ for all
$i\geq0$. Hence $\grade(\fa,N)=\grade(\fa,M)-1$ and
$\depth N =\depth M-1$. The claim follows by the
induction hypothesis.

Suppose first that
$\Ht(\fa)=n$.
 Since local cohomology modules are invariant under taking radical we may assume that $\fa$ is radical.
Hence $\fa=\fm$. Therefore,  the desired claim is in Proposition \ref{kem}.
If $\grade(\fa)<\depth R=n-m-1$, then the claim follows by Fact A). In particular,  the claim is true whenever $\Ht(\fa)<n-m-1$.
Suppose now that $\Ht(\fa)=n-m-1$. In view of Fact A we may and do assume that $\grade(\fa,R)=\Ht(\fa)$.
In particular,
 $\grade(\fa,R)=\Ht(\fa)=\cd(\fa)$. Therefore, the claim follows by definition of grade.
 Finally,
suppose  that $ n-m\leq \Ht(\fa)\leq n-1$.  Since  $\ell(\HH^i_{\fa}(R))<\infty$  for all $i<\cd(\fa)$ it follows from definition that $\f_{\fa}(R)=\Ht(\fa)$. We are going to use Grothendieck's finiteness theorem:
 $$\f_{\fa}(R)=\inf\{\depth(R_{\fq})+\Ht(\frac{\fa+\fq}{\fq}):\fq\in\Spec(R)\setminus\V(\fa)\}\quad(\ast)$$

Claim B): Let $\fq$ be any prime ideal of height  $n-1$. Then $\cmdef(R_{\fq})\leq m-1$.

Indeed,
suppose first that $\fq\in\Spec(R)\setminus\V(\fa)$. Since $\fq$ is 1-dimensional, we have $\Ht(\frac{\fa+\fq}{\fq})=\dim(\frac{R}{\fq})=1$.    In the light of $(\ast)$ we observe that
$$1+\depth(R_{\fq})=\depth(R_{\fq})+\Ht(\frac{\fa+\fq}{\fq})\geq\f_{\fa}(R)=\Ht(\fa)\geq n-m.$$
Conclude by this that $\depth(R_{\fq})\geq n-m+1$.  Hence
$$\dim(R_{\fq})-\depth(R_{\fq})\leq(n-1)- n+m-1=m-2.$$
 Second, suppose that $\fq\in\V(\fa)$.  In view
of our assumption we observe for all $i<\cd(\fa)=\Ht(\fa)$ that $\HH^{i}_{\fa R_{\fq}}(R_{\fq})\simeq \HH^{i}_{\fa}(R)_{\fq}=0$.
 This implies that $\Ht(\fa)\leq\grade(\fa_{\fq},R_{\fq})$.   We have $$n-m\leq\Ht(\fa)\leq\grade(\fa_{\fq},R_{\fq})\leq\depth(R_{\fq}),$$ because grade becomes larger with respect to inclusion. This yields that
$$\cmdef(R_{\fq})=\dim(R_{\fq})-\depth(R_{\fq})\leq(n-1)- n+m=m-1.$$
In both cases we showed $\cmdef(R_{\fq})\leq m-1$.  This completes the proof of Claim    B).

 Recall that $\cmdef(R)=m+1>m$. By Discussion \ref{dis} we have $0<\dim\left(\loc(\cmdef>m)\right)$.
In particular, there is $\fq\in\Spec(R)$ of height $n-1$ such that  $\cmdef(R_{\fq})> m$. This is a contradiction with Claim B).
This contradiction shows that any ideal $\fa$ of height in the range $ n-m\leq \Ht(\fa)\leq n-1$ disregards
the hypothesis of the  lemma.

To see the particular case we remarked that $\grade(\fa,R)\leq\f_{\fa}(R)\leq\Ht(\fa)\leq\cd(\fa).$ By the first part, $\grade(\fa,R)=\cd(\fa)$. From these $\grade(\fa,R)=\Ht(\fa)=\cd(\fa)=\f_{\fa}(R).$
This is what we want to prove.
\end{proof}

One can not replace the finite length assumption with artinian in the above lemma:
It is enough to look at a non Cohen-Macaulay invariant ring $R$ and recall that $\HH^i_{\fm}(R)$ is artinian for any $i$.

\begin{lemma}\label{l}
Let $M$ be an artinian module and let $x\in \fm$. Then
$M_x=0$.
\end{lemma}

\begin{proof}
Suppose first that $M$ is of finite length
an let $n=\ell(M)+1$. Then $x^nM=0$. From this $\frac{m}{1}=0$ for all $m\in M$. So,
$M_x=0$. In general, write $M\simeq{\varinjlim}_i M_i$ where $M_i\subset M$ is finitely generated. Then $\ell(M_i)<\infty$.
Recall that  $M_x\simeq{\varinjlim}_i(M_i)_x$. Since $(M_i)_x=0$, we get that $M_x=0$.
\end{proof}

A Krull domain with torsion classical group is called \textit{almost factorial}.

\begin{fact}\label{all4}(Samuel)
Let
$G\to \GL(n, \mathbb{F})$ be a representation of a finite group $G$. Denote the invariant ring by $R$.
Then $R$ is almost factorial.
\end{fact}

\begin{proof}
Recall that classical group of $R$ is a subgroup of $\Hom(G,\mathbb{F}^\ast)$.
Let  $f\in\Hom(G,\mathbb{F}^\ast)$ and $g\in G$. By definition, $f^{|G|}(g)=f(g^{|G|})=f(1_G)=1_\mathbb{F}$.  Also,
 a noetherian normal domain is a Krull domain. So, $R$ is almost factorial.
\end{proof}

\begin{corollary}\label{ll4}
Let
$G\to \GL(4, \mathbb{F})$ be a representation of a finite group $G$. Denote the invariant ring by $R$.
 Let $\fa\lhd_h R$ be such that $\ell(\HH^i_{\fa}(R))<\infty$ for  some $i$. Then $\HH^i_{\fa}(R)=0$.
\end{corollary}

\begin{proof}
Recall from Corollary \ref{nal} that $4=\dim R \geq\depth R\geq 3$.
First we deal with the case
$\Ht(\fa)=4$.
 Since local cohomology modules are invariant under taking radical we may assume $\fa$ is radical.
Hence $\fa=\fm$. Recall that  $\HH^{4}_{\fm}(R)$ is not finitely generated. Then,
the only crucial $\HH^j_{\fm}(R)$ is $\HH^{3}_{\fm}(R)$. Suppose it is of finite length. From this
we observe that $\HH^{j}_{\fm}(R)$ is of finite length for all $j<\dim R$.
In view of Proposition \ref{kem} $\HH^{j}_{\fm}(R)=0$  for all $j<\dim R$, as claimed.

Without loss of the generality we assume that $\Ht(\fa)<4$.
Recall that $R$  is a normal domain. Note that $\fm\in\Supp(\HH_{\fa}^{\cd(\fa,R)}(R))$.
Also, if a local ring of an algebraic variety is   normal, then it is analytically irreducible. By
 Hartshorne-Lichtenbaum vanishing theorem we deduce that $\cd(\fa,R)<4$.

In this paragraph we deal with the case
$\Ht(\fa)=3$. Recall that  $3=\Ht(\fa)\leq\cd(\fa,R)< 4$, i.e.,
$\Ht(\fa)=\cd(\fa,R)$.
If $\grade(\fa,R)=3$ then $i=3$. This is not the case, because $\HH^j_{\fa}(R)$ is not finitely generated for $j=\cd(\fa,R)$.
If $\grade(\fa,R)=2$ then $i=2$. This case excluded by Fact A) in Proposition \ref{def2}.
Then we may assume that $\grade(\fa,R)=1$. Let us again revisit Fact A) in Proposition \ref{def2}.  This allow us to assume that
$i=2$. We pick   $x$  such that $\fm=\rad(\fa+xR)$.
There is the following long exact
sequence of local cohomology modules (see \cite[Proposition 8.1.2]{BS})
$$\cdots \lo \HH_{\fa R_x}^{j-1}
(R_x) \lo \HH_{\fm}^j(R)\lo \HH_{\fa}^j(R)\lo \HH_{\fa R_x}^{j}
(R_x)\lo \cdots .$$Recall that $\HH_{\fa R_x}^{i}
(R_x)\simeq \HH_{\fa R}^{i}
(R)_x$.
In view of Lemma \ref{l} $\HH_{\fa R}^{2}
(R)_x=0$.
Hence $\HH_{\fm}^2(R)\to \HH_{\fa}^2(R)\to 0$.
Recall that $\depth(R)\geq3$. By definition of depth,
$\HH_{\fm}^2(R)=0$. In view of $\HH_{\fm}^2(R)\to \HH_{\fa}^2(R)\to 0$ we get that
$\HH_{\fa}^2(R)=0$ as desired.

Next, we assume that
$\Ht(\fa)=2$. Recall that  $2=\Ht(\fa)\leq\cd(\fa,R)<4$.
Note that $\HH^j_{\fa}(R)$ is not finitely generated for $j=\cd(\fa,R)$ and for $j=\Ht(\fa)$.
From this we conclude that $i=1$. Recall that $\grade(\fa,R)\leq\Ht(\fa)$. If
$\grade(\fa,R)=2$ the claim holds by definition of grade. Then we may assume that $\grade(\fa,R)=1=i$.
In particular, $\grade(\fa,R)=1<\depth R$. This case excluded by Fact A) in Proposition \ref{def2}.

Finally, we assume that
$\Ht(\fa)=1$.  Recall from Fact \ref{all4} that $R$ is  almost factorial. This is well-known by a result of Stroch that any
height one radical ideal over almost factorial is principal up to radical (see e.g \cite[Proposition 6.8]{fg}). Thus, $\HH_{\fa}^{>1}(R)=0$.
It remains  to note that $\HH_{\fa}^{1}(R)$ is not finitely generated.
\end{proof}

\begin{example}\label{}
Let $R:=\mathbb{F}_2[x_1,\ldots,x_4]^{\mathbb{Z}/4\mathbb{Z}}$ via the assignments $x_i\mapsto x_{i+1}$ for $i<4$ and $x_4\mapsto x_1$.
 Let $\fa:=(\sum x_i,x_1x_3+x_2x_4,\sum _{i<j<k} x_ix_jx_k)$. Then $\ell(\HH^2_{\fa}(R))=\infty$. Also, $\ell(\HH^3_{\fm}(R))=\infty$.
\end{example}

\begin{proof}
Bertin proved that $R$ is not Cohen-Macaulay by showing that the parameter sequence
$\sum x_i,x_1x_3+x_2x_4,\sum _{i<j<k} x_ix_jx_k$ is not regular sequence. From this we see that
$\grade(\fa,R)=2<3=\Ht(\fa)$. In view of Corollary \ref{ll4} $\ell(\HH^2_{\fa}(R))=\infty$.
 Recall from Corollary \ref{nal} that $\depth(R)=3$. Again, Corollary \ref{ll4} implies that $\ell(\HH^3_{\fm}(R))=\infty$.
 \end{proof}

It may be interesting to give an explicit  chain of submodules of $\HH^3_{\fm}(R)$ (resp. $\HH^2_{\fa}(R)$) with simple factors.
 We remark that "$\depth(-)=r$" is not enough to deduce $\Se(r)$. We use  the Bertin's example: $\depth(R)=3$ but $R$ is not $\Se(3)$.

\section{the local results }
The reader  may have to skip the next two items.

\begin{remark}
The main difference between local case and global case of invariant rings by a finite group
is the the following:

\begin{enumerate}
\item[i)] Nagata constructed a zero-dimensional noetherian $k$-algebra $R$ and a
finite group $G$ of automorphisms of $R$ such that $R^G$ is
non-noetherian (see \cite[Introduction]{F1}).

\item[ii)] Here is a useful  criterion: If the $R$-module of K\"{a}hler differentials for $R/pR$ over $k$ is finitely
generated for all primes $p$ dividing the order of $G$, then $R^G$ is noetherian (see \cite[main result]{F1}).\end{enumerate}
\end{remark}

\begin{discussion}Here is a quick review of the notion of non-noetherian grade.
\begin{enumerate}
\item[i)]  The classical grade  of $\fa$ on a module $M$,
denoted by $\cgrade_R(\fa,M)$, is  the supremum 
lengths of all weak regular sequences on $M$ contained in $\fa$.
In the case that $\fa$ is finitely generated by generating set
$\underline{x}:=x_{1}, \cdots, x_{r}$, the $\check{C}ech$ grade of
$\fa$ on $M$ is defined by $$\Cgrade_{R}(\fa,M):=\inf\{i\in
\mathbb{N}\cup\{0\}|\HH_{\underline{x}}^{i}(M)\neq0\}.$$ For not necessarily finitely generated
ideal $\fa$ the $\check{C}ech$ grade of $\fa$ on $M$ is defined
$$\Cgrade_{R}(\fa,M):=\sup\{\Cgrade_R(\fb,M):\fb\in\Sigma\},$$where $\Sigma$ is the family of all finitely generated
subideals $\fb$ of $\fa$. 
Recall that
 $\Cgrade_{R}(\fa,M)\geq\grade_{R}(\fa,M)$. The notation $\Cdepth(R)$
stands for $\Cgrade_{R}(\fm,R)$ when $(R,\fm)$ is a quasilocal ring.
 \item[ii)]   The  Cohen-Macaulay property of non-noetherian invariant rings investigated in \cite{at}
and \cite{adt}.
\end{enumerate}
 \end{discussion}

The following  unifies (and extends) \cite[Proposition 2]{fog} and \cite[Proposition 2]{stong}.

\begin{fact}\label{corob}
Let $(R,\fm)$ be a  local ring (resp. integral domain) and let  $G$ be a  finite group  acting on $R$. Let
$a$ and $b$ be in $R^G$ be such that they $R$-sequence. Then
$a$ and $b$ is  an $R^G$-sequence. In particular, if $\depth(R)\geq2$ then $\Cdepth(R^G)\geq2$.
\end{fact}

\begin{proof}
Clearly, $a$ is regular over $R^G$.
Let $r\in R^G$ be such that $rb=ac$ for some $c\in R^G$. Since $a$ and $b$ is an $R$-sequence, there is $d\in R$ such that
$r=ad$. Then $a(db-c)=0$.  By this, $db=c$.
Let $g\in G$. Then $g(d)b=c=db$, i.e., $b(g(d)-d)=0$.
Recall that a permutation of regular sequences  is regular (this needs both local and noetherian assumption).
Since $b$ is regular, $g(d)=d$. Thus $d\in R^G$. By definition,
$a$ and $b$ is an $R^G$-sequence.

For the particular case, let
$x\in\fm$ be a regular element.
Let $\nr(x):=\prod_{g\in G}g(x)$. Clearly, $\nr(x)\in R^G$.
 One has  $g(x)$ is $R$-regular. (If not
then there is $y$ such that $g(x)y=0$.  Apply $g^{-1}$  to this to see $xg^{-1}(y)=0$. Since $x$ is regular, $g^{-1}(y)=0$. So $y=0$ and claim follows.)
Since product of regular elements is regular, we see  that
 $\nr(x)$ is   regular. Note that length of all maximal $R$-sequences
are the same (this needs both local and noetherian assumption). There is $y\in\fm$ which is regular over $\overline{R}:=R/ \nr(x)R$. Clearly,
$g(\nr(x)R)\subset \nr(x)R$. Thus $G$ acts on $\overline{R}$. Similarly,  $\nr(y)$ is regular over $\overline{R}$.
Set $a:=\nr(x)$ and $b:=\nr(y)$.
By the first part, $a$ and $b$ is an $R^G$-sequence. By Fact \ref{max} $R^G$ is quasilocal.
So, $\Cdepth(R^G)\geq2$ as claimed.
\end{proof}

\begin{proposition}\label{ss3}
Let $(R,\fm)$ be a  $3$-dimensional  local ring and let  $G$ be a finite group  acting on $R$.   If $R$ is $\Se(2)$ then $R^G$ is $\Se(2)$.
\end{proposition}

\begin{proof}Since $R$ is $\Se(2)$, $\depth(R)\geq\min\{2,\dim R\}$, i.e., $\dim R\leq\depth(R)+1$. This means that $R$ is almost Cohen-Macaulay.
In the light of the almost Cohen-Macaulay property we see that $$\grade(P,R)=\depth(PR_P)\geq\min\{2,\dim R_P\}\quad\forall P\in\Spec(R)\quad(\ast)$$
 Let $\fp$ be any prime ideal in $R^G$.
 It is well-known from
\cite[Page 324]{Bk} that $\mathcal{S}^{-1}(R^G)=(\mathcal{S}^{-1}R)^G$
for any multiplicative closed subset $\mathcal{S}$ of $R^G$. Apply this for $\mathcal{S}:=R^G\setminus\fp$. So, $(R^G)_{\fp}\cong(R_{\fp})^G$. Recall that the integral
extension preserves Krull's dimension (see \cite[Ex. 9.2]{Mat}). This shows that $$\dim R_{\fp}=\dim((R_{\fp})^G)=\dim((R^G)_{\fp})=\Ht(\fp)\quad(+)$$ Let $P\in\Spec(R)$ be such that $\Ht(\mathcal{S}^{-1}P)=\dim R_{\fp}$.
Since $P\cap\mathcal{S}=\emptyset$ we get that $P\cap R^G\subset \fp$. Recall that localization (resp. lying over) does not increase height (see \cite[Ex. 9.8]{Mat}). Thus,
$$\Ht(\fp)\stackrel{(+)}=\dim R_{\fp}=\Ht(\mathcal{S}^{-1}P)\leq\Ht(P)\leq \Ht(\fp).$$That is $\Ht(\fp)=\Ht(P)$.
Suppose first that $\Ht(\fp)=1$. So, $\Ht(P)=1$. We are going to use $(\ast)$. Let $a$ be an $R$-regular in $P$. Clearly, $\nr(a)$ is an $R^G$-regular in $P\cap R^G\subset\fp$.
Suppose now that $\Ht(\fp)>1$.
This implies that there is a regular sequence $a,b$ in $P$.
 We observed that
 $\nr(a)$ is   regular. Note that length of all maximal $R$-sequences in $P$
is the same. There is $y\in P$ which is regular over $\overline{R}:=R/ \nr(a)R$.  Similarly,  $\nr(y)\in P\cap R^G\subset\fp$ is regular over $\overline{R}$. In particular, there is an $R$-sequence of length two in $P\cap R^G$.
In view of the above fact we see there is a $R^G$-regular sequence of length two in $\fp$. Regular sequence behave nicely with respect to localization. That is
$$\Cgrade(\fp R^G_{\fp},R^G_{\fp})\geq\Cgrade(\fp ,R^G )\geq \min\{2,\dim (R^G)_{\fp}\}.$$ From this we see that $R^G$ is $\Se(2)$.
\end{proof}

\begin{remark}
The $3$-dimensional assumption were used to deduce that $R$ is almost  Cohen-Macaulay. In particular, we have:
Let $(R,\fm)$ be an almost  Cohen-Macaulay local ring and let  $G$ be a finite group  acting on $R$. If $R$ is $\Se(2)$ then $R^G$ is $\Se(2)$.
\end{remark}

From now on we assume  $(R,\fm,k)$ is a
local ring of  characteristic $p>0$ and  $G$ is a cyclic group of
order $p^nh$, $n> 0$, acting on $R$.

\begin{definition}\label{wi} Following
Fogarty, we say $G$ acts \textit{wildly} on $R$ if the following three properties hold: i)  $G$ acts trivially on $k$, ii) $G$ acts freely on $X:=\Spec(R)$, and iii)
 $R^G$ is noetherian and  $R$ is a finite  $R^G$-module.
 \end{definition}

 The following extends and simplifies  \cite[Corollaire 2.4]{elli}.

\begin{proposition}
Let $(R,\fm)$ be a    local domain and  $G$  a finite group  acting wildly on $R$.
\begin{enumerate}
\item[i)] If $R$ is $\Se(2)$ then $R^G$ is $\Se(2)$.
\item[ii)]  Suppose  $\depth(R)>1$ and $\dim R>2$. Then   $R^G$ is not $\Se(n)$ for  all $n>2$. 
\end{enumerate}
\end{proposition}

We prove the first item by the weaker  assumption: the extension $R^G\to R$ is locally flat and the assumption given by Definition  \ref{wi}(iii).

\begin{proof} By Fact \ref{max} and Definition  \ref{wi}(iii) we see  $(R^G,\fn)$ is local.

i) The case $\dim(R)\leq 3$ is true without any condition, see Proposition \ref{ss3}. We assume that $\dim R>3$.
 Let $\fp$ be any prime ideal in $R^G$ of hight bigger than 1. In view of Fact \ref{corob} we may and do assume that $\fp\neq\fn$.
 Recall that $(R^G)_{\fp}\cong(R_{\fp})^G$. Since $G$ acts freely, the extension $R^G\to R$ is\textit{ \'{e}tale} for all $\fq\neq\fn$. Conclude by this that $(R_{\fp})^G\simeq(R^G)_{\fp}\to R_{\fp}$ is flat. Going down holds for flat extensions (see \cite[Theorem 9.5]{Mat}). In view of \cite[Ex 9.8]{Mat} and \cite[Ex 9.9]{Mat} we have $$\Ht(\fp)=\Ht(\fp (R^G)_{\fp})=\Ht(Q)\quad(+)$$ for any $Q\in\Spec(R_{\fp})$ lying over $\fp (R^G)_{\fp}\in\Spec((R^G)_{\fp})$.

Let $P\in\Spec(R)$ be such that $\mathcal{S}^{-1}P$ is the maximal ideal of $R_{\fp}$, where $\mathcal{S}:=R^G\setminus\fp$. Since $P\cap\mathcal{S}=\emptyset$ we get that $P\cap R^G\subset \fp$.
 Suppose on the contrary that $P\cap R^G\subsetneqq \fp$. By  going-up \cite[Theorem 9.4(i)]{Mat} and incomparability property \cite[Theorem 9.3(ii)]{Mat},  there is $P\subsetneqq Q\in\Spec(R)$ such that $Q\cap R^G=\fp$. Then $\mathcal{S}^{-1}P\subset  \mathcal{S}^{-1}Q$.
Since $\mathcal{S}^{-1}P$ is maximal we get to a contradiction. Hence $P\cap R^G=\fp$.
We are going to apply \cite[Ex. 9.3]{Mat}: there are only a finite number of prime ideals lying over
$\fp$. In  particular, $R_{\fp}$ is semilocal. Let $\max(R_{\fp}):=\{Q_1,\ldots,Q_n\}$.
Since $R_{\fp}$ is an integral domain and semilocal we have $\Jac(R_{\fp})\neq0$. Let $0\neq x\in\Jac(R_{\fp})$.
Suppose on the contradiction that
$\Jac(R_{\fp})\subset \bigcup_{\fq\in\ass(\frac{R_{\fp}}{(x)})} \fq$.
In the light of prime avoidance $\bigcap_{i=1}^n Q_i=\Jac(R_{\fp})\subseteq \fq$ for some $\fq\in\ass(\frac{R_{\fp}}{(x)})$.
Since the intersection index is finite and $\fq$ is prime  we have $Q_i\subseteq \fq$ for some $i$.
Thus $  Q_i=\fq$, because $Q_i$ is maximal. Suppose $Q_i=\mathcal{S}^{-1}P_i$. Then, $(R_{\fp}/(x))_{Q_i}\simeq (R/(x))_{P_i}$ is of zero depth.
Therefore, $\depth( R_{P_i})=1$. We denote this by $(\dagger)$. Recall that
$$\Ht(\fp)\stackrel{(+)}=\Ht(\mathcal{S}^{-1}P_i)\leq\Ht(P_i)\leq \Ht(\fp).$$That is $1<\Ht(\fp)=\Ht(P_i)$.  This contradicts  the  $\Se(2)$ property of $R$, see $(\dag)$. This contradiction
says that  $\Jac(R_{\fp})\nsubseteq \bigcup_{\ass(\frac{R_{\fp}}{(x)})} \fq$. Let $y\in \Jac(R_{\fp})\setminus\bigcup_{\ass(\frac{R_{\fp}}{(x)})} \fq$.
 We proved that $x$ and $y$ is an $R_{\fp}$-sequence in $\Jac(R_{\fp})$.
Clearly,
 $a:=\nr(x)$ is   regular. Note that length of all maximal $R_{\fp}$-sequences in the Jacobson radical
are the same. There is $z\in \Jac(R_{\fp})$ which is regular over $\overline{R}:=R_{\fp}/ aR_{\fp}$.  Similarly,  $b:=\nr(z)\in\Jac( R_{\fp})\cap  (R_{\fp})^G$ is regular over $\overline{R_{\fp}}$.
In the light of Fact \ref{corob} we see there is an $(R_{\fp})^G$-regular sequence of length two in $\Jac( R_{\fp})\cap  (R_{\fp})^G\subset \fp$.  From this we see that
$$\grade(\fp R^G_{\fp},R^G_{\fp})\geq \min\{2,\dim (R^G)_{\fp}\},$$and so $R^G$ is $\Se(2)$.

ii)  Since  $\depth(R)>1$ we have $\depth(R^G)\leq2$, see \cite[Proposition 4]{fog}.
Let $n>2$. Combine this with the assumption   $\dim R>2$ to see $\depth(R^G_{\fn})<\min\{\dim(R^G_{\fn}),n\}$. By definition,
 $R^G$ is not $\Se(n)$.
\end{proof}

 \begin{discussion}\label{actiontransfor}  Let $R$ be a   generalized Cohen-Macaulay
local ring of characteristic $p>0$ and of dimension $d>1$.  We assume that $R^G$ is noetherian and  $R$ is a finite  $R^G$-module.
Let $\widehat{X}:=\Spec(R)\setminus\{\fm\}$.
Then $G$ acts  on $A:=\HH^0(\widehat{X},\mathcal{O}_X)$ and  $\HH^i(G, A)$  is finitely generated as an  $R^G$-module.
 \end{discussion}

\begin{proof}
We look at $0\to\frac{R}{\HH^0_{\fm}(R)}\to A\to \HH^1_{\fm}(R)\to 0\quad(+)$
 to see $A$ is finitely generated as an $R$-module. Recall that $R$ is a finite  $R^G$-module.
 From this $A$ is finitely generated as an $R^G$-module.
 Thus, $\HH^0(G, A)=A^G$ is finitely generated as an $R^G$-module.
 To show finiteness of higher group cohomology, we denote the elements of trace zero in $A$ by $E$.
 The set $E$ is an $R^G$-submodule of $A$.
Indeed, recall that $A=\HH^0(\widehat{X},\mathcal{O}_X)={\varinjlim}_n\Hom_R(\fm^n,R)$. Let $g\in G$ and $n\in \mathbb{N}$.
Take $f:\fm^n\to R$ and define $g.f:\fm^n\to R$ by the role $(g.f)(x):=g(f(g^{-1}x))$ where $x\in\fm^n$. This defines
the action of $G$ on ${\varinjlim}\Hom_R(\fm^n,R)$. Let $r\in R^G$ and $f\in E$. We need to show $rf\in E$, i.e., $\tr(rf)=0$. To this end,
recall that
$$g.(rf)(x)=g\left((rf)(g^{-1}(x))\right)
=g\left(rf(g^{-1}(x))\right)
=g(r)g\left(f(g^{-1}(x))\right)
=rg\left(f(g^{-1}(x))\right)
=rg.(f(x)).$$
We showed $g.(rf)=r(g.f)$. We denote this property by $(\ast)$. Recall that $\tr(f)=0$.
 By definition,$$\tr(rf):=\sum_{g\in G}g.(rf)\stackrel{(\ast)}=\sum_{g\in G}r(g.f)=r\sum_{g\in G}g.f=r\tr(f)=0.$$

The set $F:=(g.r-r:r\in A)$ is an $R^G$-submodule of $A$.
Indeed, let $a\in R^G$ and let $(g.f)-f\in F$ for some $f\in A$. First,
$$a(g.f)(x)=ag(f(g^{-1}(x))=g(a)g(f(g^{-1}(x))=g\left(af(g^{-1}(x))\right)=g((af)(g^{-1}(x)).$$
From this
$$
(a(g.f-f))(x)=g\left((af)(g^{-1}(x))\right)-(af)(x)=(g.(af)-af)(x).$$
That is $a(g.f-f)=g.(af)-af\in F$.
 Keep in mind that $G$ is cyclic. Recall that (see e.g. \cite[page 178]{ns}):
\begin{equation*}
\HH^i(G, A)= \left\{
\begin{array}{rl}
A^G & \  \   \   \   \   \ \  \   \   \   \   \ \text{if } i=0\\
\frac{\ker(\tr:A\to A^G)}{\langle g(a)-a\rangle} &\  \   \   \   \   \ \  \   \   \   \   \ \text{if } i\in2\mathbb{N}+1\\
\frac{A^G}{\tr(A)} & \  \   \   \   \   \ \  \   \   \   \   \ \text{if } i\in2\mathbb{N}
\end{array} \right.
\end{equation*}
Since $\HH^{2\mathbb{N}+1}(G, A)\simeq E/F$,  $\HH^{2\mathbb{N}+1}(G,A)$ is finitely generated as an $R^G$-module.

The set $\tr(A)$ is an $R^G$-submodule of $A^G$.
Indeed, let $r\in R^G$ and $f\in \tr(A)$. Let $F\in A$ be such that $f=\tr(F)$.
In view of $(\ast)$  we have $g.(rF)=r(g.F)$. By definition, $$\tr(rF):=\sum_{g\in G}g.(rF)\stackrel{(\ast)}=\sum_{g\in G}r(g.F)=r\sum_{g\in G}g.F=r\tr(F)=rf.$$
So, $rf\in\im(\tr:A\to A^G)$, as claimed.
From this we observe that  $\HH^{2\mathbb{N}}(G, A)=\frac{A^G}{\tr(A)} $ is finitely generated as an $R^G$-module.
\end{proof}

  \begin{discussion}\label{actiolocal}  Adopt the assumption of Discussion \ref{actiontransfor} and let  $L_i:=\HH^i_{\fm}(R)$.
Then $L_i$ is a $G$-module and $\HH^j(G, L_i)$  is finitely generated as an  $R^G$-module for all $i<d$.
\end{discussion}

\begin{proof}
 The assumptions implies that $(R^G,\fn)$ is local and that $\rad(\fn R)=\fm$.
Let
$\underline{x}:=x_1,\ldots,x_d$ be a system of parameter in $\fn$.
This shows that $\rad(\underline{x} R)=\fm$.
By definition,
 $\HH^i_{\fm}(R)=\HH^i_{\underline{x}}(R)$.
Let $g:R \lo R$ be an element of $G$ and $y\in R^G$. The
assignment $r/ y^n\mapsto g(r)/ y^n$ induces an $R^G$-algebra
automorphism which we denote it again by $g:R_y \lo R_y$. This induces an $R^G$-isomorphism
of the $\check{C}$ech complexes
$g:\Check{\textbf{C}}_{\bullet}(\underline{x},R)\lo
\Check{\textbf{C}}(\underline{x},R)$. Conclude by this that there is  an $R^G$-isomorphisms
$g:\HH^i(\Check{\textbf{C}}(\underline{x},R))\lo
\HH^i(\Check{\textbf{C}}(\underline{x},R)), $
and so there is  an $R^G$-isomorphisms
$g:L_i\lo
L_i$.
Recall that:
\begin{equation*}
\HH^j(G, L_i)= \left\{
\begin{array}{rl}
\HH^i_{\fm}(R)^G & \  \   \   \   \   \ \  \   \   \   \   \ \text{if } j=0\\
\frac{\ker(\tr:L_i\to L_i^G)}{\langle g(\ell)-\ell|\ell\in L_i\rangle}&\  \   \   \   \   \ \  \   \   \   \   \ \text{if } j\in2\mathbb{N}+1\\
L_i^G/\tr(L_i) & \  \   \   \   \   \ \  \   \   \   \   \ \text{if } j\in2\mathbb{N}
\end{array} \right.
\end{equation*}
Let $\ell\in L_i$ and let $r\in R$. Since $g(r\ell)=g(r)g(\ell)$ we see that $\tr(L_i)$ is an $R^G$-submodule of $L_i^G$.
Similarly, $\ker(\tr:L_i\to L_i^G)$ and $\langle g(\ell)-\ell|\ell\in L_i\rangle$ are $R^G$-submodule of $L_i^G$.
Since $R$ is generalized Cohen-Macaulay, $L_i$ is of finite length as an $R$-module for all $i<d$. Thus
$L_i$ is  finite as an $R^G$-module for all $i<d$.
From this $\frac{L_i^G}{\tr(L_i)}$ is finitely generated as an $R^G$-module. Similarly,
$\frac{\ker(\tr:L_i\to L_i^G)}{\langle g(a)-a\rangle}$  is finitely generated as an $R^G$-module.
 \end{proof}

The origin source of the next result is \cite{local} by Grothendieck. There are a lot of works motivated by this.
 For our's propose,  Fogarty's exposition is  useful. If the reader is not
family with this technology  we suggest to look at
the friendly approach by Larry Smith \cite{simp}.

\begin{fact}(Grothendieck, Ellingsrud-Skjelbred, Fogarty, L. Smith,  etcetera) \label{sp} Let $(R,\fm,k)$ be a
local ring of  characteristic $p>0$ and  $G$ is a cyclic group of
order $p^nh$, $n> 0$, acting wilding on $R$.  By $\widehat{-}$ we mean the punctured spectrum.
In view of  \cite[Proposition 4]{fog} there is   spectral sequence
$$\HH^q\left(G,\HH^p(\widehat{\Spec(R)},\mathcal{O}_{\Spec(R)})\right)\Rightarrow \HH^{p+q}(\widehat{Y},\mathcal{O}_{\Spec(R^G)}).$$
\end{fact}

\begin{proposition}\label{ob1}
Let $(R,\fm,k)$ be a $d$-dimensional
local ring of prime characteristic $p$. Let  $G$ be a cyclic group of
order $p^nh$, $n> 0$, acting wildly on $R$.
The following assertions hold:
\begin{enumerate}
\item[a)]  If  $R$ is generalized Cohen-Macaulay, then   $R^G$ is generalized Cohen-Macaulay.
\item[b)]
   If $\depth(R)>1$, then $\Ht(\im(R\stackrel{\tr}\lo R^G))=\dim R$. Also, $\im(\tr)=\fn$ provided
$(R^G,\fn)$ is  quasi-Buchsbaum and $\depth(R)>2.$\end{enumerate}
\end{proposition}

\begin{proof}
  The assumptions implies that $(R^G,\fn)$ is local.  Let $Y:=\Spec(R^G)$, $X:=\Spec(R)$ and recall that  $\widehat{-}$ stands for the punctured spectrum.
In the light of Fact \ref{sp} we see
$\HH^q(G,\HH^p(\widehat{X},\mathcal{O}_X))\Rightarrow \HH^{p+q}(\widehat{Y},\mathcal{O}_Y).$

a) By definition $\HH^{i}_{\fm}(R)$ is finitely generated for all $i<\dim R$.
We look at the exact sequence of finitely generated $R^G$-modules $0\to R^G\to R\to R/ R^G\to 0$. This induces the following
exact sequence:$$0\to\HH^0_{\fn}( R^G)\to \HH^0_{\fn}(R)\to \HH^0_{\fn}(R/ R^G)\to \HH^1_{\fn}(R^G)\to\HH^1_{\fn}(R).$$
By independence theorem, $\HH^0_{\fn}(R)\simeq\HH^0_{\fn R}(R)$ as an $R^G$-module.
Also, $\HH^0_{\fm}(R)\simeq\HH^0_{\fn R}(R)$, because $\rad(\fn R)=\fm$.
Since $\HH^0_{\fm}(R)$ is finitely generated as an $R$-module
and that $R$ is a finite  $R^G$-module we have $\HH^0_{\fm}(R)$ is finitely
generated as an $R^G$-module. From $0\to\HH^0_{\fn}( R^G)\to \HH^0_{\fn}(R)$
we see $\HH^0_{\fn}( R^G)$ is finitely generated. We may assume $d>1$.
In view of Discussion \ref{actiontransfor} and Discussion \ref{actiolocal} cohomology groups $\HH^q(G,\HH^p(\widehat{X},\mathcal{O}_X))$ are finitely generated as an $R^G$-module for all $p<d-1$ and all $q$, because
$\HH^p(\widehat{X},\mathcal{O}_X)\simeq
\HH^{p+1}_{\fm}(R)$ if $p>0$.
For
$p+q<d-1$ the $E_2^{q,p}$ is finitely generated as an $R^G$-module. The same thing holds for all of its sub-quotients.
Thus, $E_\infty^{q,p}$ is finitely generated as an $R^G$-module.
By definition, there
is a chain
$$0=H^{n+1}\subseteq H^{n}\subseteq\cdots \subseteq
H^{0}:=\HH^{p+q}(\widehat{Y},\mathcal{O}_Y)$$  such that
$H^{i}/H^{i+1}\cong E^{i,n-i}_{\infty}$ for all $i=0,\cdots,n$. It follows
that $\HH^{p+q}(\widehat{Y},\mathcal{O}_Y)$ is  finitely generated as an $R^G$-module
provided $p+q<d-1$. Thus, $\HH^i_{\fn}(R^G)$ is finitely generated  as an $R^G$-module
for all $i<d$.

b) Since $\depth(R)>1$ we have $\HH^0_{\fm}(R)=\HH^1_{\fm}(R)=0$. From this $\HH^0(\widehat{X},\mathcal{O}_X)\simeq R$ (see (+) in Discussion \ref{actiontransfor}).
We use the five-term exact sequence of the spectral sequence
$\HH^q(G,R)\Rightarrow \HH^{p+q}(\widehat{Y},\mathcal{O}_Y)$:
$$0\lo \HH^1(G, R)\lo
\HH^1(\widehat{Y}, \mathcal{O}_Y)\lo \HH^1(\widehat{X},\mathcal{O}_X)^G\lo \HH^2(G, R)
 \lo \HH^2(\widehat{Y}, \mathcal{O}_Y).$$
 Recall that  $\HH^3_{\fn}(R^G)=\HH^2(\widehat{Y}, \mathcal{O}_Y)$ is artinian. Also $\HH^1(\widehat{X},\mathcal{O}_X)=\HH^2_{\fm}(R)$ is artinian. Conclude by this that $\HH^2(G, R)$ is artinian. Since $\HH^2(G, R)=\frac{R^G}{\im(\tr)}$  we get that $$\Ht(\im(R\stackrel{\tr}\lo R^G))=\dim R.$$
Suppose $R^G$ is  quasi-Buchsbaum   and $\depth(R)>2$. Then $\HH^1(\widehat{X},\mathcal{O}_X)^G=0$. In view of Definition \ref{wi}$(i)$,   Fogarty remarked that $1\notin\im(\tr)$. This implies that $\HH^2(G, R)\neq0$. We conclude from the quasi-Buchsbaum  property that $\fn\HH^2(\widehat{Y}, \mathcal{O}_Y)=\fn\HH^3_{\fn}(R^G)=0$. Therefore, its nonzero
submodule $\HH^2(G, R)=\frac{R^G}{\im(\tr)}$ is simple. From this we get that $\fn=\im(R\stackrel{\tr}\lo R^G)$, as claimed.
\end{proof}

\begin{corollary}
Let $(R,\fm,k)$ be a $3$-dimensional Cohen-Macaulay
local ring of  characteristic $2$. Suppose $G$ is of order two and acts  wildly on $R$.
Then $(R^G,\fn)$ is Buchsbaum if and only if $\im(\tr)=\fn$.
\end{corollary}

\begin{proof}
By Fact \ref{corob}, $\depth(R^G)\geq2$. Let $r\in \ker(\tr:R\to R^G)$. Then $0=\tr(r)=\sum_{h\in G}h(a)=r+g(r)$ where $g$ is a generator of $G$.
Thus, $g(r)=-r=r$ because $\Char(R)=2$. We conclude that $\ker(\tr)=R^G$. Also, $$(g(r)-r)_{r\in R}=(g(r)+r)_{r\in R}=\im(\tr).$$
 Recall from Proposition \ref{ob1} that  $$\HH^2_{\fn}(R^G)\simeq\HH^1(G, R)
\simeq\frac{\ker(\tr:R\to R^G)}{\langle g(r)-r\rangle}\simeq \frac{R^G}{\im(\tr)}.$$ Suppose first that $R^G$ is Buchsbaum.
In particular, $\fn\HH^2_{\fn}(R^G)=0$. From this $\im(\tr)=\fn$. Conversely, assume $\im(\tr)=\fn$. Then
$\fn\HH^i_{\fn}(R^G)=0$ for all $i<\dim R^G$.  By definition, $R^G$ is quasi-Buchsbaum. Recall that
$\HH^i_{\fn}(R^G)=0$ for all $i\neq\dim R^G$ and $i\neq\depth(R^G)$.
In view of Fact \ref{qb}  we see $R^G$ is Buchsbaum.
\end{proof}

\begin{acknowledgement}
I thank Larry Smith for his kind  feedback on the earlier version of this draft.
\end{acknowledgement}


\end{document}